\documentclass[12pt]{article}

\usepackage{amsmath,amsfonts,amssymb, amstext, amsthm, tikz, graphicx}
\usepackage{mathdots}
\usepackage{fullpage} 
\usepackage[colorlinks]{hyperref} 

\newtheorem{theorem}{Theorem}[section]
\newtheorem{lemma}[theorem]{Lemma}
\newtheorem{corollary}[theorem]{Corollary}

\newtheorem{fact}[theorem]{Fact}

\theoremstyle{definition}
\newtheorem*{definition}{Definition}

\theoremstyle{remark}
\newtheorem*{remark}{Remark}
\newcommand{\field}{\mathbb{F}}
\newcommand{\rk}{\text{rk}}
\newcommand{\FS}{\text{FS}}
\newcommand{\af}{\text{af}}

\title{New upper bounds for the Erd\H{o}s-Gy\'arf\'as problem on generalized Ramsey numbers}

\author{Alex Cameron \footnote{Department of Mathematics, Vanderbilt University, Nashville, TN 37212, USA. \newline Email: \texttt{alexander.cameron@vanderbilt.edu}}
 \and Emily Heath \footnote {Department of Mathematics, University of Illinois at Urbana-Champaign, Urbana, IL 61801, USA.  Email: \texttt{eheath3@illinois.edu}}
  }\date{\today}

\begin{document}
\maketitle

\begin{abstract}
A $(p,q)$-coloring of a graph $G$ is an edge-coloring of $G$ which assigns at least $q$ colors to each $p$-clique. The problem of determining the minimum number of colors, $f(n,p,q)$, needed to give a $(p,q)$-coloring of the complete graph $K_n$ is a natural generalization of the well-known problem of identifying the diagonal Ramsey numbers $r_k(p)$. The best-known general upper bound on $f(n,p,q)$ was given by Erd\H{o}s and Gy\'arf\'as in 1997 using a probabilistic argument. Since then,  improved bounds in the cases where $p=q$ have been obtained only for $p\in\{4,5\}$, each of which was proved by giving a deterministic construction which combined a $(p,p-1)$-coloring using few colors with an algebraic coloring.

In this paper, we provide a framework for proving new upper bounds on $f(n,p,p)$ in the style of these earlier constructions. We characterize all colorings of $p$-cliques with $p-1$ colors which can appear in our modified version of the $(p,p-1)$-coloring of Conlon, Fox, Lee, and Sudakov.  This allows us to greatly reduce the amount of case-checking required in identifying $(p,p)$-colorings, which would otherwise make this problem intractable for large values of $p$. In addition, we generalize our algebraic coloring from the $p=5$ setting and use this to give improved upper bounds on $f(n,6,6)$ and $f(n,8,8)$.
\end{abstract}

\section{Introduction}

Let $p$ and $q$ be positive integers such that $1 \leq q \leq \binom{p}{2}$. We say that an edge-coloring of a graph $G$ is a \emph{$(p,q)$-coloring} if any $p$-clique of $G$ contains edges of at least $q$ distinct colors. Let $f(n,p,q)$ denote the minimum number of colors needed to give a $(p,q)$-coloring of the complete graph on $n$ vertices, $K_n$.

This function $f(n,p,q)$ is known as the Erd\H{o}s-Gy\'{a}rf\'{a}s function after the authors of the first paper~\cite{erdos1997} to systematically study $(p,q)$-colorings. The majority of their work focused on understanding the asymptotic behavior of this function  as $n \rightarrow \infty$ for fixed values of $p$ and $q$. One of their primary results was a general upper bound of \[f(n,p,q) = O\left(n^{\frac{p-2}{\binom{p}{2}-q+1}}\right)\] obtained using the Lov\'{a}sz Local Lemma, while one of the main problems they left open was the determination of $q$, given a fixed value of $p$, for which $f(n,p,q) = \Omega(n^{\epsilon})$ for some constant $\epsilon$, but $f(n,p,q-1)=n^{o(1)}$. Towards this end, they found that \[n^{\frac{1}{p-2}}-1 \leq f(n,p,p) \leq cn^{\frac{2}{p-1}},\] where the lower bound is given by a simple induction argument and the upper bound is a special case of their general upper bound. However, they did not determine whether $f(n,p,p-1)=n^{o(1)}$.

In 2015, Conlon, Fox, Lee, and Sudakov \cite{CFLS}, building on work done on small cases by Mubayi and Eichhorn \cite{eichhorn2000note,mubayi1998note}, showed that $f(n,p,p-1)=n^{o(1)}$ by constructing an explicit $(p,p-1)$-coloring using very few colors. In \cite{55}, we slightly modified their coloring, which we call the \emph{CFLS coloring}, and paired it with an ``algebraic" construction to show that $f(n,5,5) \leq n^{1/3+o(1)}$. This improves on the general upper bound found by Erd\H{o}s and Gy\'{a}rf\'{a}s and comes close to matching their lower bound in terms of order of growth. Our construction built on the ideas of Mubayi in \cite{mubayi2004}, where he gave an explicit construction showing that $f(n,4,4) \leq n^{1/2+o(1)}$.

In this paper, we push these ideas further. In Section~\ref{CFLSsection}, we prove the following result.

\begin{theorem}\label{MainTheorem}
For any $p \geq 3$, there is a $(p,p-1)$-coloring of $K_n$ using $n^{o(1)}$ colors such that the only $p$-cliques that contain exactly $p-1$ distinct edge-colors are isomorphic (as edge-colored graphs) to one of the edge-colored $p$-cliques given in the definition below.
\end{theorem}

\begin{definition}
Given an edge-coloring $f:E(K_n) \rightarrow C$, we say that a subset $S \subseteq V(K_n)$ has a \emph{leftover structure under} $f$ if either $|S|=1$ or there exists a bipartition (which we will call the \emph{initial bipartition}) of $S$ into nonempty sets $A$ and $B$ for which
\begin{itemize}
\item $A$ and $B$ each have a leftover structure under $f$;
\item $f(A) \cap f(B) = \emptyset$; and
\item there is a fixed color $\alpha \in C$ such that $f(a,b)=\alpha$ for all $a \in A$ and all $b \in B$ and $\alpha \not\in f(A)$ and $\alpha \not\in f(B)$.
\end{itemize}
\end{definition}

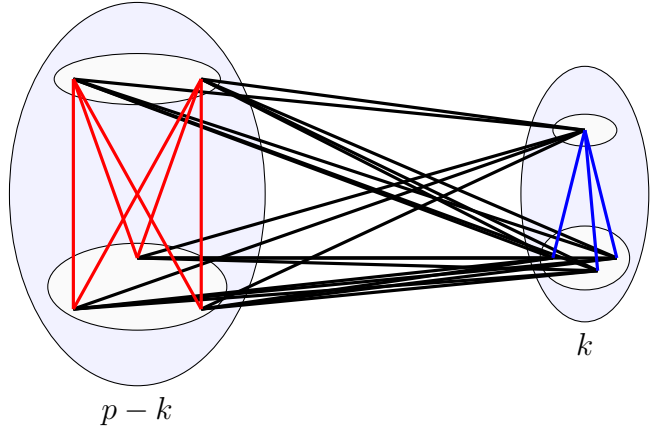
\begin{figure}
\centering
\begin{tikzpicture}[scale=.85]
		\filldraw[color=black,fill=blue!5] (0,0) ellipse (2 and 3);
		\filldraw[color=black,fill=blue!5] (7,0) ellipse (1 and 2);
		\filldraw[color=black,fill=black!2] (0,1.8) ellipse (1.3 and 0.4);
		\filldraw[color=black,fill=black!2] (0,-1.45) ellipse (1.4 and 0.68);
		\filldraw[color=black,fill=black!2] (7,1) ellipse (0.5 and 0.25);
		\filldraw[color=black,fill=black!2] (7,-1) ellipse (0.7 and 0.5);
		
		\node[below] at (0,-3) {$p-k$};
		\node[below] at (7,-2) {$k$};
		
		\draw[very thick, black] (1, 1.8) -- (7, 1);
		\draw[very thick, black] (1, 1.8) -- (6.5, -1);
		\draw[very thick, black] (1, 1.8) -- (7.5, -1);
		\draw[very thick, black] (1, 1.8) -- (7.2, -1.2);
		\draw[very thick, black] (-1, 1.8) -- (7, 1);
		\draw[very thick, black] (-1, 1.8) -- (6.5, -1);
		\draw[very thick, black] (-1, 1.8) -- (7.5, -1);
		\draw[very thick, black] (-1, 1.8) -- (7.2, -1.2);
		\draw[very thick, black] (1, -1.8) -- (7, 1);
		\draw[very thick, black] (1, -1.8) -- (6.5, -1);
		\draw[very thick, black] (1, -1.8) -- (7.5, -1);
		\draw[very thick, black] (1, -1.8) -- (7.2, -1.2);
		\draw[very thick, black] (0, -1) -- (7, 1);
		\draw[very thick, black] (0, -1) -- (6.5, -1);
		\draw[very thick, black] (0, -1) -- (7.5, -1);
		\draw[very thick, black] (0, -1) -- (7.2, -1.2);
		\draw[very thick, black] (-1, -1.8) -- (7, 1);
		\draw[very thick, black] (-1, -1.8) -- (6.5, -1);
		\draw[very thick, black] (-1, -1.8) -- (7.5, -1);
		\draw[very thick, black] (-1, -1.8) -- (7.2, -1.2);

		\draw[very thick, red] (1, 1.8) -- (1, -1.8);
		\draw[very thick, red] (1, 1.8) -- (0, -1);
		\draw[very thick, red] (1, 1.8) -- (-1, -1.8);
		\draw[very thick, red] (-1, 1.8) -- (1, -1.8);
		\draw[very thick, red] (-1, 1.8) -- (0, -1);
		\draw[very thick, red] (-1, 1.8) -- (-1, -1.8);
		
		\draw[very thick, blue] (7, 1) -- (6.5, -1);
		\draw[very thick, blue] (7, 1) -- (7.5, -1);
		\draw[very thick, blue] (7, 1) -- (7.2, -1.2);
	\end{tikzpicture}
\caption{An example of a $p$-clique with a leftover structure.}
\end{figure}

Alternatively, a more constructive definition is to say that a $p$-clique $S$ is leftover if either $p=1$ or if it can be formed from a leftover $(p-1)$-clique by taking one of its vertices $x$, making a copy $x'$, coloring $xx'$ with a new color, and coloring $x'y$ with the same color as $xy$ for each $y \in S$ for which $y \neq x$. Note that it is easy to see by induction that these $p$-cliques always contain exactly $p-1$ colors.

One of the general difficulties in producing explicit $(p,q)$-colorings is dealing with the large number of possible non-isomorphic ways to color the edges of a $p$-clique with fewer than $q$ colors in order to demonstrate that a construction avoids them.  By identifying the ``bad''  structures that are leftover after using only $n^{o(1)}$ colors, we are able to greatly reduce the amount of case-checking required in identifying $(p,p)$-colorings, which would otherwise make this problem intractable for large $p$.

More precisely, one of the nice properties of these leftover structures is that any subset of vertices of a leftover clique induces a clique that is itself leftover. Therefore, any edge-coloring of $K_n$ that eliminates leftover $p$-cliques also eliminates all leftover $P$-cliques for any $P \geq p$. Moreover, by Theorem~\ref{MainTheorem}, if this coloring uses $n^{\epsilon +o(1)}$ colors, then $f(n,P,P) \leq n^{\epsilon+o(1)}$, as the product of this coloring with the one guaranteed in Theorem~\ref{MainTheorem} will avoid any $P$-clique with fewer than $P$ colors for each $P \geq p$.

As a specific example, in \cite{55} we gave a $(5,5)$-coloring of $K_n$ that used $n^{1/3+o(1)}$ colors. Since this coloring avoids leftover 5-cliques, then it also avoids leftover $P$-cliques for all $P \geq 5$. Therefore, if we take the product of this coloring with the appropriate one developed in Section~\ref{CFLSsection} that eliminates all $6$-cliques with 5 or fewer colors other than leftover $6$-cliques, then we have a $(6,6)$-coloring that uses only $n^{1/3+o(1)}$ colors, improving the best known upper bound given above, $O(n^{2/5})$.

In Section~\ref{MDPsection}, we generalize the algebraic portion of our coloring in \cite{55}, the ``Modified Dot Product" coloring, to a version that eliminates  leftover $6$-cliques with $O(n^{1/3})$ colors (making the above example redundant) and eliminates leftover $8$-cliques with $O(n^{1/4})$ colors. By taking the product of these colorings with the appropriate ones developed in Section~\ref{CFLSsection}, this gives us the following theorem.

\begin{theorem}\label{NewBounds}
We have the following upper bounds:
\[f(n,6,6) = n^{1/3 + o(1)}; \quad f(n,8,8) = n^{1/4 + o(1)}.\]
\end{theorem}

This improves the best-known upper bound $f(n,8,8) = O(n^{2/7})$ as well.

\section{Modified CFLS coloring}\label{CFLSsection}

In this section, we define an edge-coloring $\psi_p$ of the complete graph with vertex set $\{0,1\}^{\alpha}$ for some positive integer $\alpha$. This construction is the product of two colorings, $\psi_p = c_p \times \Delta_p$, where $c_p$ is the $(p+3,p+2)$-coloring defined in \cite{CFLS}. In many places, this section tracks parts of the proof given in \cite{CFLS}, and we have attempted to keep the notation consistent with that paper to make cross-referencing easier.

We will prove the following lemma about the coloring $c_p$.

\begin{lemma}\label{CFLSstructure}
Let $p$ be a fixed positive integer. Any subset $S \subseteq \{0,1\}^{\alpha}$ with $|S| \leq p+3$ vertices that contains exactly $|S|-1$ distinct colors under the edge-coloring $c_p$ either has a leftover structure under $c_p$ or contains a striped $K_4$ under $c_p$.
\end{lemma}

A striped $K_4$, as described by the following definition, was first defined in \cite{mubayi2004}.

\begin{definition}
Let $f:E(G) \rightarrow C$ be an edge-coloring of a graph $G$. We call any 4-clique of $G$, $\{a,b,c,d\} \subseteq V(G)$, for which $f(ab)=f(cd)$, $f(ac)=f(bd)$, $f(ad)=f(bc)$, $f(ab) \neq f(ac)$, $f(ab) \neq f(ad)$, and $f(ac) \neq f(ad)$ a \emph{striped $K_4$}.
\end{definition}

We will also prove the following result about the coloring $\psi_p$.

\begin{lemma}\label{NoStripedK4}
There is no striped $K_4$ under the edge-coloring $\psi_p$.
\end{lemma}

These two lemmas are enough to conclude that $\psi_p$ is a $(p+3,p+2)$-coloring for which any clique $S$ with $|S| \leq p+3$ that contains exactly $|S|-1$ colors must have a leftover structure.

\subsection{The construction}

For some positive integer $p$, let \[1 \leq r_1 \leq r_2 \leq \cdots \leq r_p\] be fixed positive integers such that $r_d | r_{d+1}$ for each $d=1,\ldots,p-1$. These $r_i$ will be called the parameters of our edge-coloring.

For any $\alpha \geq r_p$, let $n=2^{\alpha}$, and associate each vertex of the complete graph $K_n$ with its own unique binary string of length $\alpha$. For each $d=1,\ldots,p$, let $\alpha = a_dr_d+b_d$ for positive integers $a_d,b_d$ such that $1 \leq b_d \leq r_d$. For each string $x \in \{0,1\}^{\alpha}$, we let \[x=\left(x_1^{(d)},x_2^{(d)},\ldots,x_{a_d+1}^{(d)}\right)\] where $x_i^{(d)}$ denotes a binary string in $\{0,1\}^{r_d}$ for each $i=1,\ldots,a_d$ and $x_{a_d+1}^{(d)}$ denotes a binary string from $\{0,1\}^{b_d}$. We will call these substrings $r_d$-blocks of $x$, including the final one which may or may not actually have length equal to $r_d$.

In the following definitions, we let $r_0=1$ and $r_{p+1}=\alpha$. First, we define a function $\eta_d$ for any $d=0,\ldots,p$ on domain $\{0,1\}^{\beta} \times \{0,1\}^{\beta}$ where $\beta$ is any positive integer as \[\eta_d(x,y) = \left\{ \begin{array}{ll}
\left(i,\{x_i^{(d)},y_i^{(d)}\}\right) & \quad x \neq y\\
0 & \quad x=y
\end{array} \right. \]
where $i$ denotes the minimum index for which $x_i^{(d)} \neq y_i^{(d)}$.

For $x,y \in \{0,1\}^{\alpha}$ and $0 \leq d \leq p$, let \[\xi_d(x,y) = \left(\eta_d\left(x_{1}^{(d+1)},y_{1}^{(d+1)}\right),\ldots,\eta_d\left(x_{a_{d+1}+1}^{(d+1)},y_{a_{d+1}+1}^{(d+1)}\right)\right).\] And let \[c_p(x,y) = \left(\xi_p(x,y),\ldots,\xi_0(x,y)\right).\]

Next, we assume that the binary strings of $\{0,1\}^{\beta}$ are lexicographically ordered for every positive integer $\beta$. For $1 \leq i \leq a_p+1$ and binary strings $x<y$, define \[\delta_{p,i}(x,y) = \left\{ \begin{array}{ll}
+1 & \quad \text{if } x_i^{(p)} \leq y_i^{(p)}\\
-1 & \quad \text{if } x_i^{(p)} > y_i^{(p)}. \end{array} \right. \] Let \[\Delta_p(x,y) = \left(\delta_{p,1}(x,y),\ldots,\delta_{p,a_p+1}(x,y)\right).\]

Finally, let \[\psi_p(x,y) = \left(c_p(x,y),\Delta_p(x,y)\right).\]

\subsection{Number of colors}

For any positive integer $n$, let $\beta$ be the positive integer for which \[2^{(\beta-1)^{p+1}}<n \leq 2^{\beta^{p+1}}.\] For each $d=1,\ldots,p+1$, let $r_d=\beta^d$ in the construction of $\psi_p$. Specifically, this means we are constructing the coloring on the complete graph with vertex set $\{0,1\}^{\alpha}$ where $\alpha=\beta^{p+1}$. We can apply this coloring to $K_n$ by arbitrarily associating each vertex of $K_n$ with a unique binary string from $\{0,1\}^{\alpha}$ and taking the induced coloring.

As shown in \cite{CFLS}, for these choices of parameters $r_d$, the coloring $c_p$ uses at most $2^{4(p+1)\beta^p\log_2 \beta}$ colors. On the other hand, $\Delta_p$ uses \[2^{a_p+1} \leq 2^{\beta}\] colors. So all together, $\psi_p$ uses at most $2^{4(p+1)\beta^p\log_2 \beta+\beta}$ colors, where \[(\log_2n)^{1/(p+1)} \leq \beta < (\log_2n)^{1/(p+1)} +1.\] Thus, for any fixed $p$, $\psi_p$ uses a total of $n^{o(1)}$ colors.

\subsection{Refinement of functions}

Before we prove Lemma~\ref{CFLSstructure}, it will be helpful to give the following definition and results about refinement of functions. The definition and Lemma~\ref{RefinementFact1} are paraphrased from \cite{CFLS}.

\begin{definition}
Let $f:A \rightarrow B$ and $g:A \rightarrow C$. We say that \emph{$f$ refines $g$} if $f(a_1)=f(a_2)$ implies that $g(a_1)=g(a_2)$ for all $a_1,a_2 \in A$.
\end{definition}

\begin{lemma}[Lemma 4.1(vi) from \cite{CFLS}]
\label{RefinementFact1}
Let $f,g$ be functions on domain $A$. If $f$ refines $g$, then for all $A' \subseteq A$, we have $|f(A')| \geq |g(A')|$.
\end{lemma}

\begin{lemma}\label{RefinementStructure}
Let $f,g$ be functions on domain $A$. If $f$ refines $g$ and $S \subseteq A$ is a finite subset for which $|f(S)|=|g(S)|$, then \[f(s_1)=f(s_2) \iff g(s_1)=g(s_2)\] for all $s_1,s_2 \in S$.
\end{lemma}

\begin{proof}
The forward direction follows from the definition of $f$ refining $g$. Conversely, if we have $g(s_1)=g(s_2)$ but $f(s_1) \neq f(s_2)$ for some $s_1,s_2 \in S$, then $|f(S)| \geq |g(S)|+1$, a contradiction. 
\end{proof}

In particular, Lemma~\ref{RefinementStructure} implies that if some edge-coloring of a clique $S$ is refined by another edge-coloring, but $S$ contains the same number of colors under each, then the edge-colorings must be isomorphic.

\subsection{Proof of Lemma~\ref{CFLSstructure}}

Let $S \subseteq \{0,1\}^{\alpha}$ be a set of $|S| \leq p+3$ vertices which contains exactly $|S|-1$ distinct edge colors under $c_p$. We will prove that $S$ either has a leftover structure or contains a striped $K_4$ by induction on $\alpha$, similar to the proof of Theorem 2.2 from \cite{CFLS}.

For the base case, consider $\alpha \leq r_p$. Then for any $x,y \in S$, the first component of $c_p(x,y)$ is \[\xi_p(x,y)=\left(\eta_p(x,y)\right)=\left((1,\{x,y\})\right).\] Therefore, all of the edges of $S$ receive distinct colors. So it must be that $|S|-1=\binom{|S|}{2}$, which happens only when $|S|=1,2$. In either case, $S$ trivially has a leftover structure.

Now assume that $\alpha > r_p$ and that the statement is true for shorter binary strings. For each $d=1,\ldots,p$, let $\alpha_d$ be the largest integer strictly less than $\alpha$ that is divisible by $r_d$. For any $x \in S$, let $x=(x'_d,x''_d)$ for $x'_d \in \{0,1\}^{\alpha_d}$ and $x''_d \in \{0,1\}^{\alpha-\alpha_d}$.

Let $S_d$ denote the set of $\alpha_d$-prefixes of $S$, \[S_d=\left\{x_d' \in \{0,1\}^{\alpha_d} | \exists x \in S, x=(x_d',x_d'') \right\}.\] For each $x_d' \in S_d$, let \[T_{x_d'} = \{x \in S | x=(x_d',x_d'')\}.\] Let $\Lambda_I^{(d)}$ be the set of colors contained in $S$ found on edges that go between vertices from two distinct $T$-sets, \[\Lambda_I^{(d)}=\{c_p(x,y) | x,y \in S; x_d' \neq y_d'\}.\] Similarly, let $\Lambda_E^{(d)}$ denote the set of colors contained in $S$ found on edges between vertices from the same $T$-set, \[\Lambda_E^{(d)} = \{c_p(x,y)| x,y \in S; x \neq y; x_d' = y_d'\}.\] Note that these sets of colors, $\Lambda_I^{(d)}$ and $\Lambda_E^{(d)}$, partition all of the colors contained in $S$. Therefore, \[|S|-1=|\Lambda_I^{(d)}|+|\Lambda_E^{(d)}|.\]

Next, define
\begin{align*}
    C_I^{(d)} &= \left\{(c_p(x_d',y_d'),\eta_{d-1}(x_d'',y_d''))|x,y \in S; x_d' \neq y_d'\right\}\\
    C_E^{(d)} &= \left\{\{x_d'',y_d''\}|x,y \in S; x \neq y; x_d' = y_d'\right\}.\\
\end{align*}
It is shown in \cite{CFLS} that $|\Lambda_I^{(d)}| \geq |C_I^{(d)}|$ and that $|\Lambda_E^{(d)}| \geq |C_E^{(d)}|$. The second inequality is easier to see since any distinct $x,y \in S$ for which $x_d' = y_d'$ give $\xi_d=\left(0,\ldots,0,(i,\{x_d'',y_d''\})\right)$ as the appropriate component of $c_p(x,y)$. Although the first inequality seems intuitively true, its proof is a bit more subtle. The following Fact~\ref{RefinementFact2} (proved in \cite{CFLS}) together with Lemma~\ref{RefinementFact1} give us the desired inequality. 

\begin{fact}[Lemma 4.3 from \cite{CFLS}]
\label{RefinementFact2}
For $x,y \in \{0,1\}^{\alpha}$, let \[\gamma_d(x,y)=(c_p(x_d',y_d'),\eta_{d-1}(x_d'',y_d'')).\] Then $c_p$ refines $\gamma_d$ as functions on domain $\{0,1\}^{\alpha} \times \{0,1\}^{\alpha}$.
\end{fact}

We will also use the following Fact~\ref{RefinementFact3} which is proven in \cite{CFLS}, although not stated as a claim or lemma that can be easily cited. (See the final sentence of the second-to-final paragraph on page 11.)

\begin{fact}[proved in \cite{CFLS}]\label{RefinementFact3}
There exists an integer $1 \leq d \leq p$ for which \[|C_I^{(d)}|+|C_E^{(d)}| \geq |S|-1.\]
\end{fact}

Therefore, \[|S|-1=|\Lambda_I^{(d)}|+|\Lambda_E^{(d)}| \geq |C_I^{(d)}|+|C_E^{(d)}| \geq |S|-1,\] which implies that \[|S|-1=|\Lambda_I^{(d)}|+|\Lambda_E^{(d)}| = |C_I^{(d)}|+|C_E^{(d)}|.\]

Let \[\tilde{c}_p(x,y) = \left\{ \begin{array}{ll}
(c_p(x_d',y_d'),\eta_{d-1}(x_d'',y_d'')) & \quad \text{if } x_d' \neq y_d'\\
\{x_d'',y_d''\} & \quad \text{otherwise}. \end{array} \right.\] Then by Fact~\ref{RefinementFact2} we know that $\tilde{c}_p$ refines $c_p$. And since $|\Lambda_I^{(d)}|+|\Lambda_E^{(d)}| = |C_I^{(d)}|+|C_E^{(d)}|$, then by Lemma~\ref{RefinementStructure} we know that the structure of $S$ under $\tilde{c}_p$ must be the same as the structure of $S$ under $c_p$. Therefore, we need only show that $S$ either has a leftover structure or contains a striped $K_4$ under $\tilde{c}_p$ to complete the proof. We consider two cases: either there exists some $\omega \in C_E^{(d)}$ that appears more than once in $S$ under $\tilde{c}_p$ or each $\omega \in C_E^{(d)}$ appears exactly once in $S$ under $\tilde{c}_p$.

\emph{Case 1}: Let $\omega \in C_E^{(d)}$ appear on at least two edges in $S$. This implies that $\omega=\{x_d'',y_d''\}$ and so there must exist $a,b,c,e \in S$ such that $a=(x_d',x_d'')$, $b=(x_d',y_d'')$, $c=(y_d',x_d'')$, and $e=(y_d',y_d'')$ for some $x_d' \neq y_d'$. Therefore,
\begin{align*}
    \tilde{c}_p(a,b)=\tilde{c}_p(c,e)&=\{x_d'',y_d''\}\\
    \tilde{c}_p(a,c)=\tilde{c}_p(b,e)&=(c_p(x_d',y_d'),0)\\
    \tilde{c}_p(a,e)=\tilde{c}_p(b,c)&=(c_p(x_d',y_d'),\eta_{d-1}(x_d'',y_d'')),
\end{align*}
and all three colors are distinct. Hence, $S$ contains a striped $K_4$ under $\tilde{c}_p$.

\emph{Case 2}: If each $\omega \in C_E^{(d)}$ appears exactly once in $S$ under $\tilde{c}_p$, then we know that \[|C_E^{(d)}| = \sum_{x_d' \in S_d}\binom{|T_{x_d'}|}{2}\] since each edge within a given $T$-set receives a unique color. Moreover, if we let \[C_B^{(d)} = \{c_p(x_d',y_d') | x_d',y_d' \in S_d\},\] then we know that \[|C_I^{(d)}| \geq |C_B^{(d)}| \geq |S_d|-1.\] Therefore,
\begin{align*}
    |S_d|-1+\sum_{x_d' \in S_d}\binom{|T_{x_d'}|}{2} &\leq |S|-1\\
    \sum_{x_d' \in S_d}\binom{|T_{x_d'}|}{2} &\leq |S|-|S_d|\\
    \sum_{x_d' \in S_d}\binom{|T_{x_d'}|}{2} &\leq \sum_{x_d' \in S_d} (|T_{x_d'}|-1)\\
     \sum_{x_d' \in S_d}(|T_{x_d'}|-1)(|T_{x_d'}|-2) &\leq 0.
\end{align*}
Hence, $|T_{x_d'}|=1,2$ for each $x_d' \in S_d$. This implies that $|C_E^{(d)}|=\sum_{x_d' \in S_d} (|T_{x_d'}|-1)$ and $|C_I^{(d)}|=|C_B^{(d)}|=|S_d|-1$. So by induction, $S_d$ either has a leftover structure or contains a striped $K_4$ under $c_p$. Furthermore, the coloring defined by \[c'_p(x,y) = \left\{ \begin{array}{ll}
c_p(x_d',y_d') & \quad \text{if } x_d' \neq y_d'\\
\{x_d'',y_d''\} & \quad \text{otherwise} \end{array} \right.\] is refined by $\tilde{c}_p$, and $S$ contains exactly $|S|-1$ colors under both $c'_p$ and $\tilde{c}_p$. So by Lemma~\ref{RefinementStructure}, the edge-coloring of $S$ under $\tilde{c}_p$ is isomorphic to the one under $c'_p$, and hence it is sufficient to show that $S$ has either a leftover structure or contains a striped $K_4$ under $c'_p$.

If $S_d$ has a leftover structure under $c_p$, then we see that $S$ also has a leftover structure under $c'_p$ since we can form $S$ under $c'_p$ from $S_d$ under $c_p$ by a sequence of splits as described in the definition of a leftover structure. That is, for each $x_d' \in S_d$ for which $|T_{x_d'}| = 2$, we replace $x_d'$ with two vertices
with a new edge color between them, and the same edge colors that $x_d'$ already had to the rest of the vertices.

On the other hand, if $S_d$ contains a striped $K_4$ under $c_p$, then $S$ must contain a striped $K_4$ under $c_p'$ with colors entirely from $C_B^{(d)}$. This concludes the proof.

\subsection{Proof of Lemma~\ref{NoStripedK4}}

Let $a,b,c,d \in \{0,1\}^{\alpha}$ be four distinct vertices, and assume towards a contradiction that they form a striped $K_4$ under $\psi_p$. Specifically, assume that $\psi_p(a,b)=\psi_p(c,d)$, $\psi_p(a,c)=\psi_p(b,d)$, and $\psi_p(a,d)=\psi_p(b,c)$.

Without loss of generality, we may assume the following: that $a$ is the minimum element of the four under the lexicographic ordering of $\{0,1\}^{\alpha}$; that for some $i \leq j,k$,
\begin{align*}
    \eta_p(a,b) = \eta_p(c,d) &= (i,\{x,y\})\\
    \eta_p(a,c) = \eta_p(b,d) &= (j,\{z,w\})\\
    \eta_p(a,d) = \eta_p(b,c) &= (k,\{s,t\});
\end{align*}
and that $a_i^{(p)}=c_i^{(p)}=x$ while $b_i^{(p)}=d_i^{(p)}=y$. It follows from the ordering that $x < y$ and that $a<c<b,d$. Furthermore, we have $i<j$ since  $a$ and $c$ do not differ in the $i^{th}$ block. Similarly, we see that $(k,\{s,t\})=(i,\{x,y\})$. Without loss of generality, we may let $a_j^{(p)}=b_j^{(p)}=z$ and $c_j^{(p)}=d_j^{(p)}=w$. Therefore, $z<w$ and $a<c<b<d$.

Now, it follows that $\delta_j(a,d) = +1$ and that $\delta_j(c,b)=-1$, a contradiction since we assume that $\psi_p(a,d)=\psi_p(c,b)$.

\section{Modified Dot Product coloring}\label{MDPsection}

Fix an odd prime power $q$ and a positive integer $d$. In this section, we prove Theorem~\ref{NewBounds} by giving an edge-coloring $\varphi_d$ for the complete graph on $n=(q-1)^d$ vertices that uses $(3d+1)q-1$ colors and contains no leftover 6-cliques when $d=3$ and no leftover 8-cliques when $d=4$. 

In what follows, we make use of several standard concepts and results from linear algebra without providing explicit definitions or proofs. We highly recommend \emph{Linear Algebra Methods in Combinatorics} by L\'{a}szl\'{o} Babai and P\'{e}ter Frankl \cite{linear} for a detailed treatment of these ideas. In particular, Chapter 2 covers all of the necessary background for our argument.

\subsection{The construction}

Let $\field_q^*$ denote the nonzero elements of the finite field with $q$ elements, and let $(\field_q^*)^d$ denote the set of ordered $d$-tuples of elements from $\field_q^*$. In other words, $(\field_q^*)^d$ is the set of $d$-dimensional vectors over the field $\field_q$ without zero components. In what follows, we will assume that the set $\field_q^*$ is endowed with a linear order which can be arbitrarily chosen. We then order the set $(\field_q^*)^d$ with lexicographic ordering based on the order applied to $\field_q^*$.

Define a set of colors $C_d$ as the disjoint union \[C_d = \text{DOT} \sqcup \text{ZERO} \sqcup \text{UP} \sqcup \text{DOWN},\] where $\text{DOT}=\field_q^*$, and ZERO, UP, and DOWN are each copies of the set $\{1,\ldots,d\} \times \field_q$. Let \[\varphi_d:\binom{(\field_q^*)^d}{2} \rightarrow C_d\] be a coloring function of pairs of distinct vectors, $x<y$, defined by \[\varphi_d\left(x,y\right) = \left\{ \begin{array}{ll}
(i,x_i+y_i)_{\text{ZERO}} & \quad \text{if } x \cdot y=0\\
(i,x_i+y_i)_{\text{UP}} & \quad \text{if } x \cdot y \neq 0 \text{ and } x \cdot y = x \cdot x\\
(i,x_i+y_i)_{\text{DOWN}} & \quad \text{if } x \cdot y \not\in \{0,x \cdot x\} \text{ and } x \cdot y = y \cdot y\\
x \cdot y & \quad \text{otherwise}\\
\end{array} \right. \] where $i$ is the first coordinate for which $x=(x_1,\ldots,x_d)$ differs from $y=(y_1,\ldots,y_d)$ and $x \cdot y$ denotes the standard inner product (dot product).

\subsection{Number of colors}

Let $n$ be a positive integer. Let $q$ be the smallest odd prime power for which $n \leq (q-1)^d$. Then we can color the edges of $K_n$ by arbitrarily associating each vertex with a unique vector from $(\field_q^*)^d$ and taking the coloring induced by $\varphi_d$. By Bertrand's Postulate, $q \leq 2(n^{1/d}+1)$. Therefore, the number of colors used by $\varphi_d$ on $K_n$ is at most \[(3d+1)q-1 \leq (6d+2)n^{1/d}+(6d+1).\]

\subsection{Definitions and lemmas}

\begin{definition}
Given a subset of vectors $S \subseteq \field^d$, let $\text{rk}(S)$ denote the \emph{rank} of the subset, the dimension of the linear subspace spanned by the vectors of $S$. Let $\af(S)$ denote the \emph{affine dimension} of $S$, the dimension of the affine subspace (also known as the affine hull) spanned by $S$.
\end{definition}

\begin{definition}
A color $\alpha \in C_d$ has the \emph{dot property} if $\alpha\in \text{DOT} \cup \text{ZERO}$. Note that if $\alpha$ has the dot property, then $\varphi_d\left(a,b\right) = \varphi_d\left(e,f\right)=\alpha$ implies that $a \cdot b = e \cdot f$ for any $a,b,e,f\in (\field_q^*)^d$.
\end{definition}

\begin{lemma}\label{LinIndyCondition}
Let $\{s_1,\ldots,s_t\} \subseteq (\field_q^*)^d$ be a set of $t$ linearly independent vectors and let $a,b \in (\field_q^*)^d$ such that \begin{align*}
\varphi_d\left(a,b\right)=\varphi_d\left(a,s_i\right)&=\alpha\\
\varphi_d\left(b,s_i\right)&=\beta
\end{align*} for some $\alpha,\beta \in C_d$ and for each $1\leq i \leq t$. Then $s_1,\ldots, s_{t},b$ are linearly independent.
\end{lemma}

\begin{proof}
Assume towards a contradiction that $b=\sum_{j=1}^{t} \lambda_j s_j$ for some scalars $\lambda_1,\ldots,\lambda_{t} \in \field_q$. We will first show that $\sum_{j=1}^{t}\lambda_j=1$. 

If $\alpha\in$ DOT, then $b=\sum_{j=1}^{t}\lambda_j s_j$ implies that \[\alpha=a \cdot b = \sum_{j=1}^{t}\lambda_j(a\cdot s_j)=\sum_{j=1}^{t}\lambda_j\alpha.\] Therefore, $\sum_{j=1}^{t}\lambda_j=1$ since $\alpha \notin$ ZERO.

If $\alpha\notin$ DOT, then \[a_i+b_i=a_i+s_{1,i}=\cdots=a_i+s_{t,i}\] where $i$ is the first index of difference between $a$ and $b$. Thus, $s_{j,i}=b_i$ for all $1\leq j\leq t$. But then $b=\sum_{j=1}^{t}\lambda_j s_j$ implies that \[b_i=\sum_{j=1}^{t}\lambda_j s_{j,i}=\sum_{j=1}^{t}\lambda_j b_i.\] Hence, $\sum_{j=1}^{t}\lambda_j=1$ since $b_i \neq 0$. Therefore, for any $\alpha \in C_d$ we have $\sum_{j=1}^{t} \lambda_j=1$.

Now, if $\beta$ has the dot property, then let $\beta'$ denote $b \cdot s_j$ for all $j=1,\ldots,t$. We have $$b \cdot b =\sum_{j=1}^{t}\lambda_j(b\cdot s_j)=\sum_{j=1}^{t}\lambda_j\beta'=\beta'.$$ But this implies that $\beta\in \text{UP} \cup \text{DOWN}$, contradicting that $\beta$ has the dot property.

So we must assume that $\beta$ does not have the dot property. It follows that \[b_k+s_{1,k}=\cdots=b_k+s_{t,k}\] where $k$ is the first index of difference between $b$ and $s_1$. Therefore, $s_{1,k}=\cdots=s_{t,k}$, and so \[b_k = \sum_{j=1}^t \lambda_j s_{j,k} = \sum_{j=1}^t \lambda_j s_{1,k} = s_{1,k},\] contradicting our choice of $k$.

Since we reach a contradiction for all colors $\beta$, it must be the case that $s_1,\ldots, s_{t},b$ are linearly independent vectors, as desired.
\end{proof}

We now define a particular instance of leftover structure that will be useful in our arguments.

\begin{definition}
We call the set of vectors $S=\{s_1,\ldots,s_t\} \subseteq (\field_q^*)^d$ a \emph{$t$-falling star} under the coloring $\varphi_d$ if $\varphi_d\left(s_i,s_j\right)=\alpha_i$ for all $1\leq j < i\leq t$. For any set of vectors $T \subseteq (\field_q^*)^d$ under $\varphi_d$, let $\FS(T)$ denote the maximum $t$ such that $T$ contains a $t$-falling star.
\end{definition}

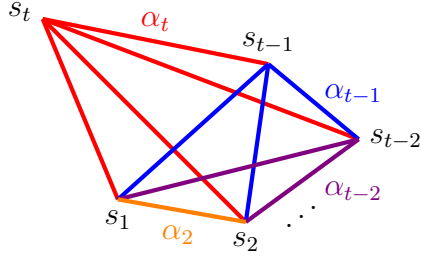
\begin{figure}
\centering
\begin{tikzpicture}
		\node[left] at (.8, 2.8) {$s_t$};
		\node[above] at (3.8, 2.1) {$s_{t-1}$};
		\node[right] at (5, 1.1) {$s_{t-2}$};
		\node[below] at (3.5, 0) {$s_2$};
		\node[below] at (1.8, .3) {$s_1$};
		\node[above, red] at (2.3, 2.4) {$\alpha_t$};
		\node[right, blue] at (4.4, 1.7) {$\alpha_{t-1}$};
		\node[right, violet] at (4.4, .4) {$\alpha_{t-2}$};
		\node[below, orange] at (2.6, .1) {$\alpha_2$};
		\node[right] at (3.85, .2) {$\iddots$};
		
		\draw[ultra thick, red] (.8, 2.7) -- (3.8, 2.1);
		\draw[ultra thick, red] (.8, 2.7) -- (5, 1.1);
		\draw[ultra thick, red] (.8, 2.7) -- (3.5, 0);
		\draw[ultra thick, red] (.8, 2.7) -- (1.8, .3);
		\draw[ultra thick, blue] (3.8, 2.1) -- (5, 1.1);
		\draw[ultra thick, blue] (3.8, 2.1) -- (3.5, 0);	
		\draw[ultra thick, blue] (3.8, 2.1) -- (1.8, .3);
		\draw[ultra thick, violet] (5, 1.1) -- (3.5, 0);
		\draw[ultra thick, violet] (5, 1.1) -- (1.8, .3);
		\draw[ultra thick, orange] (3.5,0) -- (1.8, .3);
	\end{tikzpicture}
\caption{A $t$-falling star.}
\end{figure}

The following result about these falling stars is an easy consequence of Lemma~\ref{LinIndyCondition} which can be shown by induction on the number of vectors.

\begin{corollary}\label{FallingStar}
Let $S=\{s_1,\ldots,s_t\} \subseteq (\field_q^*)^d$ be a $t$-falling star under $\varphi_d$. Then the vectors $s_1, \ldots, s_{t-1}$ are linearly independent. Consequently, for any subset $T \subseteq (\field_q^*)^d$, \[\rk(T) \geq \FS(T)-1.\] Moreover, if $T$ is contained within a monochromatic neighborhood of some other vector, then \[\rk(T) \geq \FS(T).\]
\end{corollary}

\begin{definition}
Let $A,B \subseteq \field_q^d$ be disjoint sets of vectors. We say that \emph{$A$ confines $B$} if for each $a \in A$, $a \cdot x = a \cdot y$ for every $x,y \in B$.
\end{definition}

\begin{lemma}\label{AffineSubspaces}
Let  $A,B \subseteq \field_q^d$ be disjoint sets of vectors such that $A$ confines $B$. Then \[\af(B) \leq d-\rk(A).\]
\end{lemma}

\begin{proof}
Let $t=\rk(A)$, and let $a_1,\ldots, a_t$ be linearly independent vectors from $A$. Since $A$ confines $B$, then for each $a_i$, there exists an $\alpha_i \in \field_q$ such that $a_i \cdot b = \alpha_i$ for all $b \in B$. Therefore, $B$ is a subset of the solution space for the matrix equation, 
$$\left(\begin{array}{c} -  a_1  -  \\ \vdots \\ -  a_t  - \end{array}\right)
	\left(\begin{array}{c} \mid \\ x \\ \mid \end{array}\right) 
	= \left(\begin{array}{c} \alpha_1\\ \vdots\\ \alpha_t\end{array}\right).$$
Since $a_1,\ldots, a_t$ are linearly independent, the matrix of these $t$ vectors has full rank, and hence, the solution set is an affine space of dimension $d-t$, as desired.
\end{proof}

\begin{lemma}\label{Confines}
Let $A,B \subseteq (\field_q^*)^d$ be disjoint sets of vectors and $\alpha \in C_d$ such that $\varphi_d(a,b)=\alpha$ for all $a \in A$ and $b \in B$. Then either $A$ confines $B$ or $B$ confines $A$ (or both).
\end{lemma}

\begin{proof}
If $\alpha$ has the dot property, then it is trivial that $A$ and $B$ confine one another. So assume that $\alpha \in \text{UP} \cup \text{DOWN}$. It follows that the first position of difference $i$ is the same between any $a \in A$ and any $b \in B$. Moreover, every vector of $A$ has the same $i^{th}$ component, every vector of $B$ has the same $i^{th}$ component, and every vector of $A \cup B$ has the same $j^{th}$ component for each $1 \leq j < i$ if $i>1$. Since the vectors are ordered lexicographically based on an underlying linear order of $\field_q^*$, it follows that either $a<b$ for all $a \in A$ and $b \in B$, or $b<a$ for all $a \in A$ and $b \in B$.

Without loss of generality, assume that $a<b$ for all $a \in A$ and $b \in B$. If $\alpha \in \text{UP}$, then for any particular $a \in A$, $a \cdot b = a \cdot a$ for every $b \in B$. Therefore, $A$ confines $B$. Similarly, if $\alpha \in \text{DOWN}$, then for any particular $b \in B$, $b \cdot a = b \cdot b$ for every $a \in A$,  so $B$ confines $A$.
\end{proof}

\begin{lemma}\label{NoAffineStars}
Let $t \geq 2$ be an integer. An affine subspace of $\field_q^d$ of dimension $t-2$ will contain no $t$-falling stars of $(\field_q^*)^d$ under $\varphi_d$. Therefore, \[\af(S) \geq \FS(S)-1\] for any subset of vectors $S \subseteq (\field_q^*)^d$.
\end{lemma}

\begin{proof}
We will proceed by induction on $t$. The base case $t=2$ is trivial since an affine subspace of dimension 0 is just one vector while a $2$-falling star contains two distinct vectors.

So assume that $t \geq 3$ and that the statement is true for $t-1$. Let $s_1,\ldots,s_t$ be $t$ distinct vectors that form a $t$-falling star. That is, let $\alpha_1,\ldots,\alpha_{t-1} \in C_d$ and let $\varphi_d\left(s_i,s_j\right) = \alpha_i$ when $1 \leq i < j \leq t$. Assume towards a contradiction that these vectors are contained inside an affine subspace of dimension $t-2$. Then there exist scalars $\lambda_1,\ldots,\lambda_{t-1} \in \field_q$ such that $s_t=\sum_{j=1}^{t-1} \lambda_j s_j$ and $\sum_{j=1}^{t-1} \lambda_j=1$.

First, note that if $\lambda_1=0$, then the vectors $s_2,\ldots,s_t$ form a $(t-1)$-falling star and are contained in an affine subspace of dimension $t-3$, a contradiction of the inductive hypothesis. So we must assume in what follows that $\lambda_1 \neq 0$.

Now, we consider two cases: either $\alpha_1 \in \text{DOT}$ or $\alpha_1 \not\in \text{DOT}$. If $\alpha_1 \in \text{DOT}$, then \[\alpha_1=s_1 \cdot s_t = s_1 \cdot \sum_{j=1}^{t-1} \lambda_j s_j = \lambda_1(s_1 \cdot s_1)+\alpha_1 \sum_{j=2}^{t-1} \lambda_j = \lambda_1(s_1 \cdot s_1)+\alpha_1(1-\lambda_1).\] Therefore, $\lambda_1(s_1 \cdot s_1 - \alpha_1)=0$. Since $\lambda_1 \neq 0$,  it follows that \[s_1 \cdot s_1 = \alpha_1 = s_1 \cdot s_2, \] which that implies  $\alpha_1 \not\in \text{DOT}$, a contradiction.

So assume that $\alpha_1 \not\in \text{DOT}$, and let $i$ denote the index of the first component where $s_1$ differs from the other vectors. In this case, \[s_{1,i}+s_{2,i}=\cdots=s_{1,i}+s_{t,i},\] and hence $s_{2,i}=\cdots =s_{t,i}$. Therefore, \[s_{t,i}=\sum_{j=1}^{t-1}\lambda_j s_{j,i} = \lambda_1 s_{1,i} + s_{t,i} \sum_{j=2}^{t-1} \lambda_j = \lambda_1 s_{1,i}+s_{t,i}(1-\lambda_1).\] So $\lambda_1(s_{1,i}-s_{t,i})=0$. Since $\lambda_1 \neq 0$,  we have $s_{1,i}=s_{t,i}$, a contradiction of our choice of $i$.
\end{proof}

\begin{lemma}\label{MaxStar}
Let $S \subseteq (\field_q^*)^d$ be a set of $p \geq 1$ vectors with a leftover structure under the coloring $\varphi_d$. Then \[\FS(S) \geq \left\lceil \log_2{p} \right\rceil + 1.\]
\end{lemma}

\begin{proof}
We will prove this by induction on $p$. The base case when $p=1$ is trivial, so assume that $S$ has $p \geq 2$ vectors. Then $S$ has an initial bipartition, $S=A \cup B$, and we note that \[\FS(S) \geq 1+ \max \left(\FS(A),\FS(B)\right).\] Since $|A|,|B|<p$, then by induction $\FS(T) \geq \left\lceil \log_2(|T|) \right\rceil + 1$ for $T=A,B$. Thus, we have \[\FS(S) \geq  \left\lceil \log_2\left(\max(|A|,|B|)\right) \right\rceil + 2,\] and since $\max(|A|,|B|) \geq \left\lceil \frac{p}{2} \right\rceil$, then \[FS(S)\geq \left\lceil \log_2\left(\left\lceil \frac{p}{2} \right\rceil\right) \right\rceil + 2 = \left\lceil \log_2{p} \right\rceil + 1.\]
\end{proof}

\begin{lemma}\label{WildTime}
Let $p\geq2$ and $T\geq0$ be integers. Let $S \subseteq (\field_q^*)^d$ be a subset of $p$ vectors with a leftover structure under $\varphi_d$. If $T \geq 1$, let $a_1,\ldots,a_T \in (\field_q^*)^d$ and $\alpha_1,\ldots,\alpha_T \in C_d$ such that $\varphi_d\left(a_i,a_j\right)=\alpha_i$ for all $1\leq i<j\leq T$ and $\varphi_d\left(a_i,s\right)=\alpha_i$ for all $1\leq i \leq T$ and all $s \in S$.

Then there exists a sequence of positive integers, $x_1,\ldots,x_t$ such that $\sum_{i=1}^t x_i=p-1$ and for each $i=1,\ldots,t$, the following three conditions hold:
\begin{enumerate}
    \item $1 \leq x_i \leq \left\lfloor \frac{p-s_i}{2} \right\rfloor$;
    \item $\left\lceil \log_2(x_i) \right\rceil + \left\lceil \log_2(p-s_i-x_i) \right\rceil \leq d-1$;
    \item $\left\lceil \log_2(p-s_i-x_i) \right\rceil \leq d-i-T$,
\end{enumerate}
where $s_i=0$ if $i=1$ and $s_i=\sum_{j=1}^{i-1}x_j$ otherwise.
\end{lemma}

\begin{proof}
We will prove this by induction on $p$. For the base case, let $p=2$. Let $x_1=1$ be the entire sequence. Then the first two conditions hold trivially since the sum of the sequence is 1, and since \[\left\lceil \log_2(1) \right\rceil + \left\lceil \log_2(1) \right\rceil=0 \leq d-1\] for any $d \geq 1$. For the third condition, since $\left\lceil \log_2(1) \right\rceil = 0$, it suffices to show that $T+1 \leq d$. This follows from Corollary~\ref{FallingStar}, since $S \cup \{a_1,\ldots,a_T\}$ forms a $(T+2)$-falling star, and hence $d \geq \rk\left(S \cup \{a_1,\ldots,a_T\}\right) \geq T+1$.

So assume that $S$ is a set of $p$ vertices for $p \geq 3$ and that the statement is true for smaller sets. Let the initial bipartition of $S$ be $S=A \cup B$. By Lemma~\ref{Confines}, we may assume without loss of generality that $A$ confines $B$. Therefore, $\af(B) \leq d-\rk(A)$ by Lemma~\ref{AffineSubspaces}. By Corollary~\ref{FallingStar}, we know that $\rk(A) \geq \FS(A)$ since $A$ is in a monochromatic neighborhood of any vector from $B$. And by Lemma~\ref{NoAffineStars}, we know that $\af(B) \geq \FS(B)-1$. Thus, $\FS(A)+\FS(B)-1 \leq d$. So by Lemma~\ref{MaxStar}, we can conclude that \[\left\lceil \log_2(|A|) \right\rceil + \left\lceil \log_2(|B|) \right\rceil \leq d-1.\] Therefore, setting $x_1=\min\{|A|,|B|\}$ guarantees that $1 \leq x_1 \leq \left\lfloor \frac{p}{2} \right\rfloor$ and that \[\left\lceil \log_2(x_1) \right\rceil + \left\lceil \log_2(p-x_1) \right\rceil \leq d-1.\] This gives us a positive integer $x_1$ which satisfies the first two conditions. Moreover, by Corollary~\ref{FallingStar} and Lemma~\ref{MaxStar},
\begin{align*}
    d \geq \rk\left(S \cup \{a_1,\ldots,a_T\}\right) &\geq \FS(S\cup \{a_1,\ldots, a_T\})-1\\
    & \geq (T+1+\max\left(\FS(A),\FS(B)\right))-1\\
    &\geq T + \left\lceil \log_2(p-x_1) \right\rceil +1.
\end{align*}
Thus, $x_1$ also satisfies the third condition.

Let $S'$ denote the larger of the two parts $A$ and $B$, and let $a_{T+1}$ denote an arbitrary vector from $S \setminus S'$. Then $S'$ contains $p-x_1<p$ vectors and has a leftover structure under $\varphi_d$. Moreover, $S'$ and $a_1,\ldots,a_{T},a_{T+1}$ satisfy the monochromatic neighborhood conditions of the hypothesis. Hence, by induction there exists a sequence of positive integers $x_1',\ldots,x_{t'}'$ such that $\sum_{i=1}^{t'}x_i'=p-x_1-1$ and for each $i=1,\ldots,t'$, the following three conditions hold:
\begin{enumerate}
    \item $1 \leq x'_i \leq \left\lfloor \frac{p-x_1-s'_i}{2} \right\rfloor$;
    \item $\left\lceil \log_2(x'_i) \right\rceil + \left\lceil \log_2(p-x_1-s'_i-x'_i) \right\rceil \leq d-1$;
    \item $\left\lceil \log_2(p-x_1-s'_i-x'_i) \right\rceil \leq d-i-(T+1)$,
\end{enumerate}
where $s'_i=0$ if $i=1$ and $s'_i=\sum_{j=1}^{i-1}x'_j$ otherwise.

Let $x_i=x'_{i-1}$ for $2\leq i \leq t'+1$ and let $t=t'+1$ to get a sequence $x_1,\ldots,x_t$ for which \[\sum_{i=1}^t x_i=x_1+\sum_{i=1}^{t'}x_i'=x_1+p-x_1-1=p-1.\] For each $i=2,\ldots,t$, the first two conditions are satisfied since $x_1+s'_i=s_{i+1}$, and the third condition is satisfied since $d-i-(T+1)=d-(i+1)-T$.
\end{proof}

\begin{corollary}\label{6clique}
Let $S \subseteq (\field_q^*)^3$ be a set of 6 vectors. Then $S$ cannot have a leftover structure under the coloring $\varphi_3$.
\end{corollary}

\begin{proof}
If such a set exists, then by Lemma~\ref{WildTime} with $T=0$, a positive integer $x_1$ exists such that $1 \leq x_1 \leq 3$ and \[\left\lceil \log_2(x_1) \right\rceil + \left\lceil \log_2(6-x_1) \right\rceil \leq 2.\] It is simple to check that no such integer exists.
\end{proof}

\begin{corollary}\label{8clique}
Let $S \subseteq (\field_q^*)^4$ be a set of 8 vectors. Then $S$ cannot have a leftover structure under the coloring $\varphi_4$.
\end{corollary}

\begin{proof}
If such a set exists, then by Lemma~\ref{WildTime} with $T=0$, we must be able to find a sequence of positive integers $x_1,x_2,\ldots,x_t$ that satisfy the conditions given in the Lemma. In particular, $1 \leq x_1 \leq 4$ and \[\left\lceil \log_2(x_1) \right\rceil + \left\lceil \log_2(8-x_1) \right\rceil \leq 3.\] We can check and find that $x_1=1$ is the only possibility. Therefore, $1 \leq x_2 \leq 3$ such that
\begin{align*}
\left\lceil \log_2(7-x_2) \right\rceil &\leq 2\\
\left\lceil \log_2(x_2) \right\rceil + \left\lceil \log_2(7-x_2) \right\rceil &\leq 3.\\
\end{align*}
A quick check reveals that no such integer exists.
\end{proof}

Theorem~\ref{NewBounds} follows from Theorem~\ref{MainTheorem} and Corollaries~\ref{6clique} and~\ref{8clique}.

\section{Conclusion}

The proof of Lemma~\ref{WildTime} actually shows which leftover $p$-cliques can appear under $\varphi_d$ for a particular $d$. For example, this proof implies that the only leftover $5$-clique that can appear under $\varphi_3$ is a monochromatic $C_4$ contained inside a monochromatic neighborhood of one vertex (that is, an initial $(1,4)$-bipartition with a $(2,2)$-bipartition inside the part with four vertices). In \cite{55}, we handled this specific leftover structure by splitting each color class of $\varphi_3$ into four new colors determined by certain relations between vectors. While the current paper can be viewed as our attempt to fully generalize the coloring techniques used in \cite{55} and \cite{mubayi2004}, it does not generalize the splitting that was crucial for handling the final leftover $5$-clique. Perhaps such a generalized splitting would be enough to give $f(n,p,p) \leq n^{1/(p-2) +o(1)}$ for $p \geq 6$ or at least improve the best-known upper bounds for values of $p$ other than the two  addressed in this paper.

\begin{remark} 
Since the completion of this work, Conlon, Pohoata, and Tyomkyn have emailed us that they have obtained a version of Theorem~\ref{MainTheorem} independently.
\end{remark}

\bibliography{88paper}
\bibliographystyle{plain}

\end{document}